\documentclass[10pt]{article}
\usepackage{amsmath,amsthm,amssymb,latexsym,enumerate}
\usepackage{color}
\usepackage
[draft=false,setpagesize=false,pdfstartview=FitH,
colorlinks=true,linkcolor=blue,citecolor=magenta,pagebackref=true,bookmarks=false]
{hyperref}

\newtheorem{thm}{Theorem}[section]

\newtheorem{lem}[thm]{Lemma}
\newtheorem{prop}[thm]{Proposition}
\newtheorem{rem}[thm]{Remark}

\theoremstyle{definition}

\newcommand{\R}{\Bbb R}

\newcommand{\e}{\varepsilon }
\newcommand{\RN}{\R^N}

\def\red#1{\textcolor{red}{#1}}

\allowdisplaybreaks

\numberwithin{equation}{section}

\begin{document}
\title{\bf Existence and non-existence of maximizers\\for the
Moser-Trudinger type inequalities under inhomogeneous constraints}

\author{Norihisa Ikoma\,$^1$, Michinori Ishiwata\,$^2$ and Hidemitsu Wadade\,$^3$}
\date{\small\it 
$^1$Department of Mathematics,
Faculty of Science and Technology, 
Keio University, Yokohama, Kanagawa 2238522, Japan
\vspace{.3cm}\\
$^2$Department of Systems Innovation, 
Graduate School of Engineering Science, Osaka University,\\
Toyonaka, Osaka 5608531, Japan
\vspace{.3cm}\\
$^3$Faculty of Mechanical Engineering, Institute of Science and Engineering, Kanazawa University,\\
Kanazawa, Ishikawa 9201192, Japan}

\maketitle

\begin{abstract}
In this paper, we study the existence and non-existence of maximizers 
for the Moser-Trudinger type inequalities in $\Bbb R^N$ of the form 
	\[
		D_{N,\alpha}(a,b):=
		\sup_{u\in W^{1,N}(\Bbb R^N),\,\|\nabla u\|_{L^N(\Bbb R^N)}^a+\|u\|_{L^N(\Bbb R^N)}^b=1}
		\int_{\Bbb R^N}\Phi_N\left(\alpha|u|^{N'}\right)dx.
	\]
Here $N\geq 2$, $N'=\frac{N}{N-1}$, $a,b>0$, $\alpha \in (0,\alpha_N]$ 
and $\Phi_N(t):=e^t-\sum_{j=0}^{N-2}\frac{t^j}{j!}$ 
where $\alpha_N:= N \omega_{N-1}^{1/(N-1)}$ and $\omega_{N-1}$ denotes 
the surface area of the unit ball in $\RN$. 
We show the existence of the threshold $\alpha_\ast = \alpha_\ast(a,b,N) \in [0,\alpha_N]$ 
such that 
$D_{N,\alpha}(a,b)$ is not attained if $\alpha \in (0,\alpha_\ast)$ 
and is attained if $ \alpha \in (\alpha_\ast , \alpha_N)$. 
We also provide the conditions on $(a,b)$ in order that the inequality 
$\alpha_\ast < \alpha_N$ holds.

\bigskip


\noindent 
{\it Key words\,:} Moser-Trudinger inequality, maximizing problem, Sobolev's embedding theorem, vanishing phenomenon, concentration
\end{abstract}

\section{Introduction and main results}
\noindent

The classical Moser-Trudinger inequality originates from the embedding $W^{1,N}_0(\Omega)\hookrightarrow L_{\psi_{N,\alpha}}(\Omega)$ 
for any bounded domain $\Omega\subset\Bbb R^N$ with some small $\alpha>0$ independently proved in \cite{Po, T, Y}. 
Here $N\geq 2$ and $L_{\psi_{N,\alpha}}(\Omega)$ denotes the Orlicz space associated 
with the Young function $\psi_{N,\alpha}(t):=\exp(\alpha|t|^{N'})-1$ where 
$N':=\frac{N}{N-1}$. In \cite{M}, this embedding was sharpened by proving  
\begin{align}\label{classical-MT}
&B_{N,\alpha}(\Omega):=\sup_{u\in W^{1,N}_0(\Omega),\,\|\nabla u\|_{L^N(\Omega)}=1}\int_\Omega\exp\left(\alpha |u(x)|^{N'}\right)dx\\
&\notag\begin{cases}
&<+\infty\quad\text{if \,}0<\alpha\leq \alpha_N:=N\omega_{N-1}^{\frac{1}{N-1}},\\
&=+\infty\quad\text{if \,}\alpha>\alpha_N 
\end{cases}
\end{align}
for any bounded domain $\Omega$, where $\omega_{N-1}$ denotes the surface area of the unit ball in $\Bbb R^N$. 

On the other hand, when volume of $\Omega$ is infinite, 
there are several extensions of the Moser-Trudinger inequality \eqref{classical-MT}. 
Firstly, we mention a scaling invariant version established in \cite{AT, Og} as follows:
\begin{align}\label{scaling-invari-MT}
&C_{N,\alpha}:=\sup_{u\in W^{1,N}(\Bbb R^N),\,\|\nabla u\|_{L^N(\Bbb R^N)}=1}\frac{1}{\|u\|_{L^N(\Bbb R^N)}^N}\int_{\Bbb R^N}\Phi_N\left(\alpha|u(x)|^{N'}\right)dx\\
&\notag\begin{cases}
&<+\infty\quad\text{if \,}0<\alpha<\alpha_N,\\
&=+\infty\quad\text{if \,}\alpha\geq\alpha_N,
\end{cases}
\end{align}
where $\Phi_N(t):=e^t-\sum_{j=0}^{N-2}\frac{t^j}{j!}=\sum_{j=N-1}^\infty\frac{t^j}{j!}$ for $t\geq 0$. We also refer to \cite{D} concerning the related work to \eqref{scaling-invari-MT}. Here we stress that $C_{N,\alpha_N} = + \infty$, which is different from the bounded domain case. 
Moreover, the following estimates of $C_{N,\alpha}$ as $\alpha \nearrow \alpha_N$ 
were derived in \cite{CST-14, LLZ}: 
for $\gamma \in (0,1)$, 
	\begin{equation}\label{1-3}
\frac{c_N}{1-\gamma^{N-1}}\leq
		C_{N,\gamma \alpha_N} \leq \frac{\tilde c_N}{1-\gamma^{N-1}}, 
	\end{equation}
where $c_N$ and $\tilde c_N$ are positive constants depending only on $N$. 
We also refer to \cite{KSW,NW,OgOz,O,O2} for other extensions of the Moser-Trudinger type 
inequalities in various directions.

Next, we state another extension of \eqref{classical-MT} established in \cite{LLZ, LR, R}. 
For $a,b>0$, we define the quantity $D_{N,\alpha}(a,b)$ by
\begin{align}\label{dnalpha-def}
&D_{N,\alpha}(a,b):=\sup_{u\in W^{1,N}(\Bbb R^N),\,\|\nabla u\|_{L^N(\Bbb R^N)}^a+\|u\|_{L^N(\Bbb R^N)}^b=1}\int_{\Bbb R^N}\Phi_N\left(\alpha|u(x)|^{N'}\right)dx.
\end{align}
The works \cite{LR, R} established the finiteness of $D_{N,\alpha_N}(a,b)$ for the case $N = a = b$. 
In \cite{LLZ}, the authors generalized this result by proving 
\begin{align}\label{inhomo-tm}
&D_{N,\alpha}(a,b)
\begin{cases}
&<+\infty\quad\text{if \,}0<\alpha\leq \alpha_N,\\
&=+\infty\quad\text{if \,}\alpha>\alpha_N
\end{cases}
\quad\text{when \,}b\leq N,
\end{align} 
and
\begin{align*}
&D_{N,\alpha}(a,b)
\begin{cases}
&<+\infty\quad\text{if \,}0<\alpha<\alpha_N,\\
&=+\infty\quad\text{if \,}\alpha\geq\alpha_N
\end{cases}
\quad\text{when \,}b>N. 
\end{align*}
It is worth noticing that 
when $\alpha = \alpha_N$, the finiteness of $D_{N,\alpha}(a,b)$ 
varies depending on the size of $b$. 
See also \cite{C, CST-14, LL} for other extensions similar to $D_{N,\alpha}(a,b)$.

Next we turn to the existence and non-existence of a maximizer of $B_{N,\alpha}(\Omega)$, 
$C_{N,\alpha}$ and $D_{N,\alpha}(a,b)$. In \cite{CC}, it was shown that $B_{N,\alpha_N}(\Omega)$ is attained when $\Omega$ is a ball. 
After that, in \cite{F, L}, the existence of a maximizer of $B_{N,\alpha_N}(\Omega)$ was proved for any bounded domains. 
In order to show the existence of a maximizer of $B_{N,\alpha_N}(\Omega)$, 
we need to avoid a lack of compactness caused by 
the concentration of maximizing sequences. 
For a related work, we also refer to \cite{S}. 

For proving the existence of a maximizer of  
$C_{N,\alpha}$ or $D_{N,\alpha}(a,b)$ with $\alpha<\alpha_N$, 
we need to avoid the lack of the compactness. 
In this case, concentration phenomena do not occur (see \cite[Lemma 4.2]{DST-16}) 
and vanishing phenomena are issues 
due to the unboundedness of the domain. 
Concerning the maximizing problem $D_{N,\alpha_N}(a,b)$ with $b\leq N$, 
we may also suffer from the lack of the compactness caused by the concentration.  
In \cite{INW}, the authors showed that $C_{N,\alpha}$ is attained for all $\alpha \in (0,\alpha_N)$. 
On the other hand, the situation becomes more complicated 
for the maximizing problem associated with $D_{N,\alpha}(a,b)$. 
The case $a=b=N$ was discussed in \cite{I, LR, R}. 
More precisely, the case $\alpha = \alpha_N$ and $N \geq 2$ was treated in \cite{LR,R}.
In \cite{I}, the existence of a maximizer was established in the cases $N \geq 3$, $\alpha \in (0,\alpha_N)$ and 
$N=2$, $\alpha \in (\alpha_\ast,\alpha_2]$ with some $0<\alpha_\ast<\alpha_2$.  
Moreover, the author in \cite{I} also proved the non-existence of a maximizer 
when $N=2$ and $0<\alpha \ll 1$. 
For general $a$ and $b$, the authors in \cite{DST-16} showed that $D_{N,\alpha}(a,b)$ is attained when 
$N\geq 2$, $\alpha = \alpha_N$, $a>N'$ and $0<b<N$. 

Finally, we mention the works \cite{Lam,LY-18} in which 
the authors considered the existence 
of a maximizer of the singular Moser-Trudinger inequality, 
namely, the inequality obtained by replacing $dx$ and $0 < \alpha \leq \alpha_N$ 
by $\frac{dx}{|x|^\beta}$ and 
$ 0< \alpha \leq \alpha_{N,\beta}:=(1-\frac{\beta}{N})\alpha_N$ in \eqref{dnalpha-def} 
where $0<\beta < N$. 
In \cite{Lam} the existence 
of a maximizer was proved for the case $\alpha<\alpha_{N,\beta}$ or the case $\alpha=\alpha_{N,\beta}$ with $b<N$. 
On the other hand, in \cite{LY-18}, the case $\alpha = \alpha_{N,\beta}$ with $b=N$ is dealt. 
Here we remark that the singular weight $|x|^{-\beta}$ 
compensates the lack of the compactness due to the vanishing phenomenon. 
This point is different from the non-singular case $\beta=0$.

\medskip

Motivated by the above works, 
in this paper, we extend the result in \cite{DST-16} and 
treat the case $\alpha < \alpha_N$ or the case $a \leq N'$. 
As mentioned above, we need to consider the effect of vanishing phenomenon and 
prove the existence of a threshold $\alpha_\ast \in [0,\alpha_N]$ such that 
there is no maximizer of $D_{N,\alpha}(a,b)$ if $0<\alpha < \alpha_\ast$, 
while there is a maximizer of $D_{N,\alpha}(a,b)$ if $\alpha_\ast < \alpha < \alpha_N$.  
In addition, we also provide some qualitative estimates for $\alpha_\ast$, 
and in particular we give a condition on $a$ and $b$ which yields $\alpha_\ast < \alpha_N$. 

In order to state our first result, 
for each $(a,b)\in(0,\infty)^2$, 
we define the value $\alpha_* = \alpha_\ast (a,b,N)\in[0,\alpha_N]$ by 
\begin{align*}
\alpha_*:=\inf\{
\alpha\in(0,\alpha_N]\,\,\big|\,\text{$D_{N,\alpha}(a,b)$ is attained\,}
\}
\end{align*}
when $D_{N,\alpha}(a,b)$ is attained for some $\alpha\in (0,\alpha_N]$. 
Also we set $\alpha_*=\infty$ when $D_{N,\alpha}(a,b)$ is not attained for any $\alpha\in(0,\alpha_N]$. Our first result now reads as follows. 

\begin{thm}\label{con-compactness}
\noindent
\emph{(i)} 
Let $a,b > 0$ and suppose $\alpha_\ast < \alpha_N$. 
When $b < N$, $D_{N,\alpha}(a,b)$ is attained for $\alpha_\ast < \alpha \leq \alpha_N$. 
When $b \geq N$, $D_{N,\alpha}(a,b)$ is attained for $\alpha_\ast < \alpha < \alpha_N$. 

\medskip

\noindent
\emph{(ii)}
Let $a,b > 0$ and suppose $\alpha_\ast < \alpha_N$.
The function defined by 
	\[
		\alpha \mapsto \frac{(N-1)!}{\alpha^{N-1}} D_{N,\alpha}(a,b) 
		: (\alpha_\ast, \alpha_N) \to \R
	\]
is strictly increasing. Moreover, by setting $D_{N,0}(a,b) = 0$, there holds
\begin{align*}
D_{N,\alpha}(a,b)
\begin{cases}
&=\displaystyle\frac{\alpha^{N-1}}{(N-1)!}\quad\text{for \,}\alpha\in[0,\alpha_*],\\
&>\displaystyle\frac{\alpha^{N-1}}{(N-1)!}\quad\text{for \,}\alpha\in(\alpha_*,\alpha_N), 
\end{cases}
\end{align*}
and in particular 
\begin{align}\label{11}
\alpha_*
=\inf\left\{
\alpha\in(0,\alpha_N)\,\bigg|\,D_{N,\alpha}(a,b)>\frac{\alpha^{N-1}}{(N-1)!}
\right\}.
\end{align}

\medskip

\noindent
\emph{(iii)} 
The constant $\alpha_\ast$ is non-increasing in 
$a$ and $b$, that is, if $(a_1, b_1), (a_2, b_2)\in (0,\infty)^2$ with $a_1\leq a_2$ and $b_1 \leq b_2$, it holds $\alpha_\ast (a_1,b_1,N) \geq \alpha_\ast(a_2,b_2,N)$. 
\end{thm}

Next, we give qualitative estimates on $\alpha_\ast$. To this end, we introduce $B_{GN}$ as 
the best constant of the following Gagliardo--Nirenberg inequality: 
	\[
		B_{GN} := \sup_{ u \in W^{1,N}(\RN) \setminus \{0\} } 
		\frac{ \| u \|_{L^{NN'}(\Bbb R^N)}^{NN'} }{ \| u \|_{L^N(\Bbb R^N)}^N \| \nabla u \|_{L^N(\Bbb R^N)}^{NN'-N} } \in (0,\infty). 
	\]
Then we have the following result. 

\begin{thm}\label{alpha*estimates}
\emph{(i)} 
For any $a>N'$ and $b>0$, there holds $\alpha_*=0$.

\medskip

\noindent
\emph{(ii)} 
For any $a\leq N'$ and $b>0$, there holds $\alpha_*>0$.

\medskip

\noindent
\emph{(iii)} 
Fix $\hat{\alpha} \in (0,\alpha_N)$ and $\hat{b} > \frac{N^2}{\hat{\alpha} B_{GN}}$. 
Then there exists an $a_0 = a_0(\hat{\alpha}, \hat{b}, N) \in (0,N')$ such that 
$D_{N,\alpha}(a,b)$ is attained for $(\alpha,a,b) \in A_1 \cup A_2$ 
where $A_1 :=\{ \alpha_N \} \times (a_0,N'] \times [\hat{b}, N)$ 
and $A_2 := [\hat{\alpha}, \alpha_N) \times (a_0,N'] \times [\hat{b}, \infty)$. 
In particular, $\alpha_\ast \leq \hat{\alpha} < \alpha_N$ holds 
for $(a,b) \in (a_0,N'] \times [\hat{b}, \infty)$. 
\end{thm}

\begin{rem}
(i) In Appendix \ref{appendix}, we prove the strict inequality 
$\frac{N^2}{\alpha_N B_{GN}} < N$. Therefore, when $\hat{b} \in \left(\frac{N^2}{\alpha_N B_{GN}} , N\right)$, 
we may choose $\hat{\alpha} \in (0,\alpha_N)$ satisfying 
$\frac{N^2}{\hat{\alpha} B_{GN}}<N$. Hence, taking $\hat{b}\in\left(\frac{N^2}{\hat{\alpha} B_{GN}}, N\right)$, we obtain $A_1 \neq \emptyset$ in Theorem \ref{alpha*estimates} (iii). 

\medskip

\noindent
(ii) By applying Theorem \ref{alpha*estimates} (iii), we obtain 
$\alpha_*(N',b,N)\leq\frac{N^2}{b B_{GN}}<\alpha_N$ for $b>\frac{N^2}{\alpha_N B_{GN}}$. 
Indeed, for $b>\frac{N^2}{\alpha_N B_{GN}}$ and $\hat\alpha\in \left(\frac{N^2}{b B_{GN}}, \alpha_N\right)$, Theorem \ref{alpha*estimates} (iii) with $a=N'$ gives 
$\alpha_*(N',b,N)\leq\hat\alpha$. Taking a limit $\hat{\alpha} \to \frac{N^2}{b B_{GN}}$, we have the desired inequality. 
Combining this inequality with Theorem \ref{alpha*estimates} (ii), 
we have $0< \alpha_\ast (N',b,N) \to 0$ as $b \to \infty$. 

\medskip

\noindent
(iii) By Theorem \ref{alpha*estimates} (ii)-(iii), we see that 
when $(a,b)\in(0,N']\times\left(\frac{N^2}{\alpha_NB_{GN}},\infty\right)$ 
and $a$ is close to $N'$, there holds $0<\alpha_*<\alpha_N$. 
Thus, in this case, by Theorem \ref{con-compactness}, 
$D_{N,\alpha}(a,b)$ is attained for all $ \alpha \in (\alpha_* ,\alpha_N)$ 
(resp. $\alpha \in (\alpha_*  ,\alpha_N]$\,) when $b\geq N$ (resp. $b<N$), 
while $D_{N,\alpha}(a,b)$ is not attained for all $\alpha \in (0,\alpha_\ast )$. 
\end{rem}

Finally, we give some comments on our arguments to prove 
Theorems \ref{con-compactness} and \ref{alpha*estimates}. 
We take an approach which is similar to those in \cite{DST-16,I,IW-16}. 
As mentioned above, there might be a lack of compactness for maximizing sequences 
$(u_n)_{n=1}^\infty$, that is, 
the concentration ($\| \nabla u_n \|_{L^N(\Bbb R^N)} \to 1$ as $n \to \infty$) and 
the vanishing ($\| u_n \|_{L^N(\Bbb R^N)} \to 1$ as $n \to \infty$). 
In our case, we can rule out the concentration as in \cite[Lemma 4.2]{DST-16} 
with the aid of the inequality \eqref{1-3}. 
To avoid the vanishing, we shall show that 
the strict inequality $D_{N,\alpha}(a,b) > \frac{\alpha^{N-1}}{(N-1)!}$ is sufficient. 
Then we reveal some properties of $D_{N,\alpha}(a,b)$ and prove 
the existence (or non-existence) of a maximizer depending on the size of parameters $(\alpha,a,b)$. 
Finally, we also provide the proof 
of the strict inequality $\frac{N^2}{\alpha_N B_{GN}} < N$. 

%

\section{Proof of Theorem \ref{con-compactness}}\label{sec2}
\noindent

This section is devoted to proving Theorem \ref{con-compactness}. 
The following proposition becomes a key for the proof of Theorem \ref{con-compactness}. 

\begin{prop}\label{>-compact}
\emph{(i)} 
Let $(a,b)\in(0,\infty)^2$ and $\alpha\in(0,\alpha_N)$, 
and assume $D_{N,\alpha}(a,b)>\frac{\alpha^{N-1}}{(N-1)!}$.  
Then $D_{N,\alpha}(a,b)$ is attained. 

\noindent
\emph{(ii)} Let $\alpha = \alpha_N$, $a>0$ and $b < N$. 
Assume $D_{N,\alpha_N}(a,b) > \frac{\alpha_N^{N-1}}{(N-1)!}$. 
Then $D_{N,\alpha_N}(a,b)$ is attained. 
\end{prop}

\noindent
In order to prove Proposition \ref{>-compact}, we prepare several lemmas. 
Take a maximizing sequence $\{u_n\}_{n\in\Bbb N}\subset W^{1,N}(\Bbb R^N)$ associated with $D_{N,\alpha}(a,b)$, 
i.e., the functions $\{u_n\}_{n\in\Bbb N}$ satisfy $\|\nabla u_n\|_N^a+\|u_n\|_N^b=1$ and 
$\int_{\Bbb R^N}\Phi_N(\alpha|u_n|^{N'}) dx \to D_{N,\alpha}(a,b)$ as $n\to\infty$. 
Then up to a subsequence, there exists some $u_0\in W^{1,N}(\Bbb R^N)$ such that $u_n\rightharpoonup u_0$ weakly in $W^{1,N}(\Bbb R^N)$ 
as $n\to\infty$. We may assume that the functions $\{u_n\}_{n\in\Bbb N}$ are non-negative, 
radially symmetric and non-increasing in the radial direction by virtue of the radially symmetric rearrangement. 
We start with the following fact and a similar statement is proved  in \cite[Lemma 5.1]{DST-16}. 
\begin{lem}\label{compact-lem1}
Assume either \emph{(i)} $\alpha < \alpha_N$, $(a,b) \in (0,\infty)^2$ 
or else \emph{(ii)} $\alpha = \alpha_N$, $a>0$ and $0 < b \leq N$. 
Let $\{u_n\}_{n=1}^\infty \subset W^{1,N}(\RN)$ be a radial 
maximizing sequence for $D_{N,\alpha}(a,b)$ with 
$u_n \rightharpoonup u_0$ weakly in $W^{1,N}(\RN)$. 
In addition, when $\alpha = \alpha_N$, suppose 
$\limsup_{n \to \infty} \| \nabla u_n \|_N < 1$. 
Then as $n \to \infty$, 
\begin{align*}
&\int_{\Bbb R^N} \Psi_N( \alpha |u_n|^{N'}  ) dx 
\to 
\int_{\Bbb R^N} \Psi_N(\alpha |u_0|^{N'} ) dx
\end{align*}
where 
	\[
		\Psi_N(s) := \Phi_N(s) - \frac{s^{N-1}}{(N-1)!} 
		= \sum_{j=N}^\infty \frac{s^j}{j!}. 
	\]
\end{lem}

\begin{proof}
This follows from Strauss' lemma (see \cite{BL-83-1,S-77} ). In fact, 
when the assumption (i) holds, put 
	\[
		P(s) := \Psi_N(\alpha |s|^{N'}) 
		= \sum_{j=N}^\infty \frac{(\alpha |s|^{N'} )^j }{j!}, \quad 
		Q(s) := \Phi_N(\beta |s|^{N'})
	\]
where $\alpha < \beta < \alpha_N$. 
We first notice that 
	\[
		\lim_{|s| \to 0} \frac{P(s)}{Q(s)} = 0 = \lim_{|s| \to \infty} \frac{P(s)}{Q(s)}.
	\]
Second, we shall claim
	\begin{equation}\label{1}
		\int_{\RN} \Phi_N \left( \beta |u(x)|^{N'} \right) dx 
		\leq C_\beta \| u \|_{N}^N \quad 
		{\rm for\ all} \ u \in W^{1,N}(\RN) \quad {\rm with} \ \| \nabla u \|_{N} \leq 1.
	\end{equation}
Indeed, when $N=2$, this is proved in \cite[Lemma 1]{BJT-08} and we follow 
the argument there for $N\geq 2$. Since \eqref{1} clearly holds for $u \equiv 0$, 
we assume $u \not \equiv 0$. 
From  \cite{AT,LR,R}, we notice that 
	\begin{equation}\label{2}
		\| \nabla v \|_{N}^N \int_{\RN} \Phi_N \left( \beta 
		\left|\frac{v(x)}{\| \nabla v \|_{N}}\right|^{N'} \right) 
		dx \leq C_\beta \| v \|_{N}^N \quad 
		{\rm for\ all} \ v \in W^{1,N}(\RN) \setminus \{0\}.
	\end{equation}
If $ u \in W^{1,N}(\RN) \setminus \{0\}$ satisfies $\| \nabla u \|_{N} \leq 1$, 
then from $-N + N'j \geq 0$ for all $j \geq N-1$, we have  
	\[
		\begin{aligned}
			\Phi_N \left( \beta |u(x)|^{N'} \right) 
			&= \sum_{j=N-1}^\infty \frac{\beta^j}{j!} |u(x)|^{N'j}
			\\
			&\leq \sum_{j=N-1}^\infty \frac{\beta^j}{j!} 
			\frac{|u(x)|^{N'j}}{\| \nabla u \|_{N}^{-N + N'j}  }
			\\
			&= \| \nabla u \|_{N}^N 
			\sum_{j=N-1}^\infty \frac{\beta^j}{j!} 
			\frac{|u(x)|^{N'j}}{\| \nabla u \|_{N}^{N'j}  }
			\\
			&=\| \nabla u \|_{N}^N 
			\Phi_N \left( \beta \left| \frac{u(x)}{\|\nabla u \|_{N}} \right|^{N'} \right).
		\end{aligned}
	\]
Thus by \eqref{2}, we obtain \eqref{1}. 

	Now from \eqref{1} and $ \| \nabla u_n \|_N \leq 1$, it is obvious to see 
	\[
		\sup_{n \geq 1} \int_{\RN} Q(|u_n|) d x < +\infty.
	\]
Furthermore, since $u_n$ is radial, we see that
$| u_n(x) | \leq C |x|^{-\frac{N-1}{N}}  $ holds for every $n \geq 1$ and $|x| \geq 1$, 
where $C$ is independent of $n$. Therefore, applying Strauss' lemma, we get 
	\[
		\int_{\RN} P(u_n) dx \to \int_{\RN} P(u_0) d x.
	\]

	On the other hand, when the assumption (ii) holds, put 
	\[
		P(s) := \Phi_N \left(\alpha_N |s|^{N'} \right), \quad 
		Q(s) := \Phi_N \left((1+\e_0)^{N'} \alpha_N |s|^{N'} \right), 
	\]
where $\e_0>0$ is chosen so that 
$(1+2 \e_0) \| \nabla u_n \|_N \leq 1$. 
Since 
	\[
		(1 + \e_0)^{N'} \alpha_N |u_n|^{N'} 
		= \left\{  \left( \frac{1+\e_0}{1+2\e_0} \right)^{N'} \alpha_N \right\}
		| (1+2\e_0) u_n |^{N'}, \quad 
		\left( \frac{1+\e_0}{1+2\e_0} \right)^{N'} \alpha_N < \alpha_N,
	\]
we observe from \eqref{1} that 
	\[
		\sup_{n \geq 1} \int_{\RN} Q(|u_n|) dx
		= \sup_{n \geq 1} \int_{\RN} 
		\Phi_N \left( \left\{ \frac{1+\e_0}{1+2\e_0} \right\}^{N'} \alpha_N 
		\left| (1+2\e_0) u_n \right|^{N'}  \right) dx
		 < +\infty.
	\]
The rest of the proof is same to the case (i), 
hence Lemma \ref{compact-lem1} holds. 
\end{proof}

\begin{lem}\label{une0-max}
Assume the same conditions to Lemma \ref{compact-lem1}. 
In addition, suppose $u_0 \not\equiv 0$. 
Then $D_{N,\alpha}(a,b)$ is attained by $u_0$. 
\end{lem}

A similar assertion is also proved in \cite[Theorem 1.1 and Lemma 2.2]{DST-16}

\begin{proof}
Since $u_0\ne 0$ in $W^{1,N}(\Bbb R^N)$, let $\tau_n:=\frac{\|u_n\|_N}{\|u_0\|_N}$ and 
$\tau_0 := \lim_{n \to \infty} \tau_n$. Remark that $\tau_0 \geq 1$ holds 
due to $u_n \rightharpoonup u_0$ weakly in $W^{1,N}(\RN)$.

	We first show $\tau_0 = 1$. For this purpose, we set 
$v_0(x) := u_0(\tau^{-1}_0 x)$. Since 
	\begin{equation}\label{21}
		\begin{aligned}
			\| \nabla v_0 \|_{N}^a + \| v_0 \|_N^b 
			= \| \nabla u_0 \|_{N}^a + \tau_0^b \| u_0 \|_N^b 
			&= \| \nabla u_0 \|_N^a + \lim_{n \to \infty} \tau_n^b \| u_0 \|_{N}^b 
			\\
			&\leq \liminf_{n\to\infty} \left( \| \nabla u_n \|_N^a + \| u_n \|_N^b \right) = 1,
		\end{aligned}
	\end{equation}
we see 
	\[
		D_{N,\alpha}(a,b) \geq \int_{\RN} \Phi_N(\alpha |v_0|^{N'}) dx
		= \frac{\tau_0^N \alpha^{N-1}}{(N-1)!} \| u_0 \|_{N}^N 
		+ \tau_0^N \int_{\RN} \Psi_N(\alpha |u_0|^{N'}) dx.
	\]
If $\tau_0>1$, then by $u_0 \neq 0$ and Lemma \ref{compact-lem1}, one has 
	\[
		\begin{aligned}
			D_{N,\alpha}(a,b) &> \frac{\tau_0^N \alpha^{N-1}}{(N-1)!} \| u_0 \|_{N}^N 
			+  \int_{\RN} \Psi_N(\alpha |u_0|^{N'}) dx
			\\
			&= \lim_{n \to \infty}
			 \frac{ \alpha^{N-1}}{ (N-1)! }\tau_n^N  \| u_0 \|_{N}^N 
			+  \lim_{n \to \infty} \int_{\RN} \Psi_N(\alpha |u_n|^{N'}) dx
			\\
			&= \lim_{n\to\infty} \int_{\RN} \Phi_N( \alpha |u_n|^{N'} ) dx = D_{N,\alpha}(a,b),
		\end{aligned}
	\]
which is a contradiction. Thus, $\tau_0 = 1$ and $\| u_n \|_N \to \|u_0\|_{N}$.

	Next, if $\| \nabla u_0 \|_N < \liminf_{n \to \infty} \| \nabla u_n \|_{N} \leq 1$, 
then from $ \tau_n \to 1$ and \eqref{21}, 
we may find a $t_0>1$ so that 
	\[
		\| \nabla t_0 u_0 \|_N^a + \| t_0u_0 \|_N^b = 1.
	\]
Since $\Phi_N(s)$ is strictly increasing in $s$, $u_0 \not\equiv 0$ and 
$\| u_n \|_N \to \|u_0\|_N$, it follows from Lemma \ref{compact-lem1} that 
	\[
		\begin{aligned}
			D_{N,\alpha}(a,b) \geq \int_{\RN} \Phi_N(\alpha |t_0u_0|^{N'}) dx 
			&> \int_{\RN} \Phi_N(\alpha |u_0|^{N'}) dx 
			\\
			&= \lim_{n \to \infty} \int_{\RN} 
			\Phi_N(\alpha |u_n|^{N'}) dx = D_{N,\alpha}(a,b),
		\end{aligned}
	\]
which is a contradiction again. Therefore, $\| \nabla u_n \|_{N} \to \| \nabla u_0 \|_N$.

	Since $W^{1,N}(\RN)$ is uniformly convex and $\| u_n \|_{W^{1,N}} \to \|u_0 \|_{W^{1,N}}$, 
we see that $u_n \to u_0$ strongly in $W^{1,N}(\RN)$ and complete the proof. 
\end{proof}

\begin{lem}\label{vn-level}
\emph{(i)} Let $(a,b)\in(0,\infty)^2$ and $\alpha\in(0,\alpha_N)$. 
If there exists a radial maximizing sequence $\{u_n\}_{n=1}^\infty \subset W^{1,N}(\RN)$ 
such that $u_n \rightharpoonup 0$ weakly in $W^{1,N}(\RN)$, 
then $D_{N,\alpha}(a,b) = \frac{\alpha^{N-1}}{(N-1)!}$ holds. 

\medskip

\noindent
\emph{(ii)} Let $\alpha=\alpha_N$, $0<a$ and $b \leq N$. 
Assume $\{u_n\}_{n=1}^\infty \subset W^{1,N}(\RN)$ is 
a radial maximizing sequence for $D_{N,\alpha_N}(a,b)$ 
with $\limsup_{n \to \infty} \| \nabla u_n \|_N < 1$ and 
$u_n \rightharpoonup 0$ weakly in $W^{1,N}(\RN)$. 
Then $D_{N,\alpha_N}(a,b) = \frac{\alpha^{N-1}_N}{(N-1)!}$. 
\end{lem}

\begin{proof}
We prove (i) and (ii) at the same time. 
By Lemma \ref{compact-lem1}, the assumption $u_n \rightharpoonup 0$ weakly in $W^{1,N}(\Bbb R^N)$ 
and the normalization $\|u_n\|_N \leq(\|\nabla u_n\|_N^a+\|u_n\|_N^b)^{\frac{1}{b}}\leq 1$, we see 
	\[
		\begin{aligned}
			\int_{\Bbb R^N}\Phi_N\left(\alpha|u_n|^{N'}\right)
			&= \frac{\alpha^{N-1}}{(N-1)!} \| u_n \|_N^N 
			+ \int_{\RN} \Psi_N( \alpha |u_n|^{N'}) dx 
			\\
			&= \frac{\alpha^{N-1}}{(N-1)!} \| u_n \|_N^N + o(1) 
			\leq \frac{\alpha^{N-1}}{(N-1)!} + o(1). 
		\end{aligned}
	\]
Letting $n\to\infty$ in the above, we obtain 
$D_{N,\alpha}(a,b) \leq \alpha^{N-1} / (N-1)!$.

	In order to prove the opposite inequality $D_{N,\alpha}(a,b)\geq \alpha^{N-1} / (N-1)!$, 
let $ 0< \e < 1$ and $\varphi \in W^{1,N}(\RN)$ satisfy $ 1 - \e \leq \| \varphi \|_N^N < 1$. 
For $t>0$, put $\varphi_t (x) := t \varphi( t x)$. Then it is easy to see that 
$\| \nabla \varphi_t \|_N^a + \| \varphi_t \|_{N}^b \leq 1$ 
for all sufficiently small $t>0$. Furthermore, by Lemma \ref{compact-lem1} and $\Psi_N(s) \geq 0$, 
for sufficiently small $t>0$, 
	\[
		\begin{aligned}
			D_{N,\alpha}(a,b) \geq 
			\int_{\RN} \Phi_N(\alpha |\varphi_t|^{N'}) dx 
			&= \frac{\alpha^{N-1}}{(N-1)!} \| \varphi \|_{N}^N 
			+ \int_{\RN} \Psi_N( \alpha |\varphi_t|^{N'} ) d x  
			\\
			&\geq 
			\frac{\alpha^{N-1}}{(N-1)!} \| \varphi \|_N^N 
			\geq \frac{\alpha^{N-1}}{(N-1)!} ( 1 - \e).
		\end{aligned}
	\]
Since $\e \in (0,1)$ is arbitrary, we have $D_{N,\alpha}(a,b) \geq \frac{\alpha^{N-1}}{(N-1)!}$. 
Hence, 
we obtain the desired conclusion. 
\end{proof}

\begin{rem}
	From the above argument, it is easily seen that 
		\begin{equation}\label{22}
			D_{N,\alpha}(a,b) \geq \frac{\alpha^{N-1}}{(N-1)!}
		\end{equation}
	for all $\alpha \in (0,\alpha_N]$ and $(a,b) \in (0,\infty)^2$. 
\end{rem}

Now we are in the position to prove Proposition \ref{>-compact}. 

\begin{proof}[\bf{Proof of Proposition \ref{>-compact}}]
Let $\{u_n\}_{n=1}^\infty \subset W^{1,N}(\RN)$ be a radial maximizing sequence and 
$u_n \rightharpoonup u_0$ weakly in $W^{1,N}(\RN)$.

	(i) In view of Lemma \ref{une0-max}, 
it is enough to show $u_0 \ne 0$ in $W^{1,N}(\Bbb R^N)$. 
On the contrary, assume $u_0=0$ in $W^{1,N}(\Bbb R^N)$. 
Then Lemma \ref{vn-level} yields $D_{N,\alpha}(a,b)\leq\frac{\alpha^{N-1}}{(N-1)!}$, 
which is a contradiction to the assumption.

	(ii) In this case, we need to show $\limsup_{n\to\infty} \| \nabla u_n \|_N < 1$. 
If $\| \nabla u_n \|_N \to 1$, then by \cite[Lemma 4.2]{DST-16}, we see 
	\[
		\int_{\RN} \Phi_N(\alpha_N |u_n|^{N'} ) dx \to 0 = D_{N,\alpha_N}(a,b),
	\]
which is a contradiction and $\limsup_{n \to \infty} \| \nabla u_n \|_N < 1$ holds. 
From this fact, we can use a similar argument to case (i) and 
Proposition \ref{>-compact} holds. 
\end{proof}

For the proof of Theorem \ref{con-compactness}, 
we need the following lemma. 

\begin{lem}\label{con-alpha-inc2}
Let $(a,b)\in(0,\infty)^2$. 

\medskip

\noindent
\emph{(i)} The function $\frac{(N-1)!}{\alpha^{N-1}}D_{N,\alpha}(a,b)$ is non-decreasing for $0<\alpha \leq \alpha_N$. 

\medskip

\noindent
\emph{(ii)} 
Let $0<\beta_1<\beta_2 \leq \alpha_N$. 
Suppose that $b < N$ if $\beta_2 = \alpha_N$, and that 
$D_{N,\beta_1}(a,b)$ is attained.  
Then 
	\[
		\frac{(N-1)!}{\beta_2^{N-1}} D_{N,\beta_2} (a,b) 
		> \frac{(N-1)!}{\beta_1^{N-1}} D_{N,\beta_1} (a,b) 
	\]
and $D_{N,\beta_2}(a,b)$ is also attained. \textsl{}
\end{lem}

\begin{proof}
(i)  Since 
	\[
		\frac{(N-1)!}{\alpha^{N-1}} \Phi_{N}( \alpha |s|^{N'} ) 
		= (N-1)! \sum_{j=N-1}^\infty \frac{\alpha^{j-N+1} (|s|^{N'})^j }{j!},
	\]
it is immediate to see that 
	\begin{equation}\label{23}
		\frac{(N-1)!}{\beta_1^{N-1}} \Phi_{N} ( \beta_1 |s|^{N'} ) 
		< \frac{(N-1)!}{\beta^{N-1}_2} \Phi_{N}(  \beta_2 |s|^{N'}  ) 
		\quad {\rm for\ all} \ s \neq 0, \ 0 < \beta_1 < \beta_2 \leq \alpha_N.
	\end{equation}
Therefore, (i) holds.

(ii) Since $D_{N,\beta_1}(a,b)$ is attained by the assumption, 
there exists $u_0\in W^{1,N}(\Bbb R^N)$ such that 
$\|\nabla u_0\|_N^a+\|u_0\|_N^b=1$ and 
$D_{N,\beta_1}(a,b)=\int_{\Bbb R^N}\Phi_N(\beta_1|u_0|^{N'})$. 
Then it follows from \eqref{22} and \eqref{23} that 
	\[
		\begin{aligned}
			\frac{(N-1)!}{\beta_2^{N-1}}D_{N,\beta_2}(a,b) 
			& \geq \frac{(N-1)!}{\beta_2^{N-1}}
			\int_{\Bbb R^N}\Phi_N\left(\beta_2|u_0|^{N'}\right) dx
			\\
			&> \frac{(N-1)!}{\beta_1^{N-1}}\int_{\Bbb R^N}\Phi_N\left(\beta_1|u_0|^{N'}\right) dx
			= \frac{(N-1)!}{\beta_1^{N-1}}D_{N,\beta_1}(a,b) \geq 1. 
		\end{aligned}
	\]
Hence, there holds $D_{N,\beta_2}(a,b)>\frac{\beta_2^{N-1}}{(N-1)!}$ 
and Proposition \ref{>-compact} asserts that $D_{N,\beta_2}(a,b)$ is attained. 
\end{proof}

We are now ready to prove Theorem \ref{con-compactness}. 

\begin{proof}[\bf Proof of Theorem \ref{con-compactness}]
Let $a,b > 0$ and assume $\alpha_\ast (a,b,N) < \alpha_N$.

(i) By the definition of $\alpha_\ast$ and Lemma \ref{con-alpha-inc2}, 
it is clear that for any $\alpha \in (\alpha_\ast , \alpha_N)$, 
$D_{N,\alpha}(a,b)$ is attained. 
In addition, if $b < N$, then we also observe that 
$D_{N,\alpha}(a,b)$ is attained for every $\alpha \in (\alpha_\ast, \alpha_N]$ 
thanks to Lemma \ref{con-alpha-inc2}.

(ii) Again by Lemma \ref{con-alpha-inc2}, 
the function $\alpha \mapsto (N-1)! D_{N,\alpha}(a,b) / \alpha^N$ 
is strictly increasing in $(\alpha_\ast,\alpha_N)$.

We shall show $D_{N,\alpha_\ast}(a,b) = \alpha_\ast^{N-1} / (N-1)!$. 
From our convention for the case $\alpha_\ast=0$, we may assume $\alpha_\ast \in (0,\alpha_N)$. 
By contradiction, suppose $D_{N,\alpha_\ast} (a,b) > \alpha_\ast^{N-1} / (N-1)!$ due to \eqref{22}. 
Since $\alpha_\ast< \alpha_N$, we infer from Proposition \ref{>-compact} that 
$D_{N,\alpha_\ast}(a,b)$ is attained and let $u_0$ be a maximizer for 
$D_{N,\alpha_\ast} (a,b)$. Noting
	\[
		\lim_{\alpha \to \alpha_\ast} \int_{\RN} \Phi_N(\alpha |u_0|^{N'}) dx 
		= \int_{\RN} \Phi_N(\alpha_\ast |u_0|^{N'}) dx = D_{N,\alpha_\ast}(a,b)
		> \frac{\alpha_\ast^{N-1}}{(N-1)!}, 
	\]
if $\alpha \in (0,\alpha_\ast)$ is sufficiently close to $\alpha_\ast$, then 
we obtain
	\[
		D_{N,\alpha}(a,b) \geq \int_{\RN} \Phi_N(\alpha |u_0|^{N'}) dx 
		> \frac{\alpha_\ast^{N-1}}{(N-1)!} > \frac{\alpha^{N-1}}{(N-1)!}.
	\]
For such an $\alpha \in (0,\alpha_\ast)$, Proposition \ref{>-compact} asserts that 
$D_{N,\alpha}(a,b)$ is attained and this contradicts the definition of $\alpha_\ast$. 
Hence, we have $D_{N,\alpha_\ast} (a,b) = \alpha_\ast^{N-1} / (N-1)!$. 
It is also easy to see that \eqref{11} holds.

Finally, \eqref{22}, Proposition \ref{>-compact} and the definition of $\alpha_\ast$ yield 
$D_{N,\alpha}(a,b) = \alpha^{N-1} / (N-1)!$ for each $\alpha \in [0,\alpha_\ast]$.

(iii) If $a_1 \leq a_2$, $b_1 \leq b_2$ and 
$\| \nabla u \|_N^{a_1} + \| u \|_N^{b_1} \leq 1$, 
then $\| \nabla u \|_{N}^{a_2} + \| u \|_{N}^{b_2} \leq 1$, hence, 
$D_{N,\alpha}(a_1,b_1) \leq D_{N,\alpha}(a_2,b_2)$. 
Thus from \eqref{11}, we see $\alpha_\ast (a_1,b_1,N) \geq \alpha_\ast(a_2,b_2,N)$ 
and Theorem \ref{con-compactness} holds. 
\end{proof}

\section{Proof of Theorem \ref{alpha*estimates}}\label{sec3}
\noindent

In this section, we shall show Theorem \ref{alpha*estimates}. 
We first deal with the case (i) in Theorem \ref{alpha*estimates}, that is, the case $a>N'$. 

\begin{prop}\label{thm-proof-i-iii}
	Let $(a,b)\in(N',\infty)\times(0,\infty)$ and $\alpha \in (0,\alpha_N)$. 
	Then $D_{N,\alpha}(a,b)$ is attained, and hence, there holds $\alpha_\ast (a,b,N) = 0$. 
\end{prop}

\begin{proof}
By virtue of Proposition \ref{>-compact}, 
it suffices to show $D_{N,\alpha}(a,b)>\frac{\alpha^{N-1}}{(N-1)!}$.

Let $v\in W^{1,N}(\Bbb R^N)\setminus\{0\}$ with 
$\| \nabla v \|_N^a + \| v \|_N^b = 1$ 
and set $v_t(x):=t^{1/N}v(t^{1/N} x)$ for $t>0$. 
Then we see for $\beta>0$, 
\begin{align*}
&\|\nabla(\beta v_t)\|_N^a+\|\beta v_t\|_N^b=\beta^a\|\nabla v_t\|_N^a+\beta^b\|v_t\|_N^b
=\beta^a t^{a/N}\|\nabla v\|_N^a+\beta^b\|v\|_N^b.
\end{align*}
Thus,  we find a unique $\beta_*(t)>0$ such that $w_t:=\beta_*(t) v_t$ satisfies 
\begin{align}\label{beta*t-formula}
\|\nabla w_t\|_N^a+\|w_t\|_N^b=\beta_*(t)^a t^{a/N}\|\nabla v\|_N^a+\beta_*(t)^b\|v\|_N^b=1. 
\end{align}

	By the definition of $D_{N,\alpha}(a,b)$, we have 
	\begin{align*}
	D_{N,\alpha}(a,b) \geq\int_{\Bbb R^N}\Phi_N(\alpha|w_t|^{N'})
	&=\sum_{j=N-1}^\infty\frac{\alpha^j}{j!}\|w_t\|_{N'j}^{N'j}
	\\
	&=\sum_{j=N-1}^\infty\frac{\alpha^j}{j!}\beta_*(t)^{N'j}\|v_t\|_{N'j}^{N'j}\\
	&\geq\sum_{j=N-1}^N\frac{\alpha^j}{j!}\beta_*(t)^{N'j}t^{\frac{j}{N-1}-1}\|v\|_{N'j}^{N'j}
	\\
	&=\frac{\alpha^{N-1}}{(N-1)!}\left(
	\beta_*(t)^N\|v\|_N^N+\frac{\alpha}{N}\beta_*(t)^{NN' }t^{1/(N-1) }\|v\|_{NN' }^{ NN'}
	\right)\\
	&=:\frac{\alpha^{N-1}}{(N-1)!}f(t). 
	\end{align*}

	On the other hand, recalling \eqref{beta*t-formula}
one sees that 
	\begin{equation}\label{38}
	\lim_{t \downarrow 0} \beta_\ast (t) = \|v\|_{N}^{-1}, \quad 
	\beta_\ast'(t) =
	-\frac{\frac{a}{N}t^{\frac{a}{N}-1}\beta_*(t)^a\|\nabla v\|_N^a}
	{
		a\beta_*(t)^{a-1}t^{\frac{a}{N}}\|\nabla v\|_N^a+b\beta_*(t)^{b-1}\|v\|_N^b
	}, \quad 
	\lim_{t \downarrow 0} f(t) = 1.
	\end{equation}
	Therefore, in order to prove the strict inequality 
	$D_{N,\alpha}(a,b)>\frac{\alpha^{N-1}}{(N-1)!}$, 
	it suffices to prove $f'(t)>0$ for small $t>0$. 
	We set
	\begin{align}
	\notag g(t) :=&
	\frac{\alpha}{N(N-1)}\|v\|_{ NN'}^{NN'}\beta_*(t)^{NN'}
	+ \alpha N'  \|v\|_{NN' }^{NN'} 
	t\,\beta_*(t)^{NN'-1}\beta_*'(t)\\
&\label{39}+N\|v\|_N^N\,t^{1-\frac{1}{N-1}}\beta_*(t)^{N-1}\beta_*'(t), 
	\end{align}
	and then a direct computation shows $f'(t) = t^{-1 + 1/(N-1)} g(t)$. 
	Thus, $f'(t) > 0$ is equivalent to showing $g(t) > 0$.

	From \eqref{38} and $a>N'$, it is easy to see that 
	\[
	t^{1-\frac{1}{N-1}} \beta_\ast'(t) \to 0 \quad 
	{\rm as} \ t \to 0.
	\]
	Hence, by \eqref{39}, we see that $g(t)$ is continuous and $g(0) > 0$. 
	Thus, one has $f'(t) > 0$ for sufficiently small $t>0$ and 
	Proposition \ref{thm-proof-i-iii} holds. 
\end{proof}

	Next, we consider the case $a \leq N'$. 
First we show $\alpha_\ast > 0$ as follows. 

\begin{prop}\label{small-non-exist}
Let $(a,b)\in(0,N']\times(0,\infty)$. 
Then there exists $\alpha_0\in(0,\alpha_N)$ depending on $N$, 
$a$ and $b$ such that $D_{N,\alpha}(a,b)$ is not attained for all $\alpha\in(0,\alpha_0)$. 
\end{prop}

\begin{proof}
Let $(a,b)\in(0,N']\times(0,\infty)$ and assume that $D_{N,\alpha}(a,b)$ is attained by some function $v\in W^{1,N}(\Bbb R^N)$ 
with $\|\nabla v\|_N^a+\|v\|_N^b=1$ for some $\alpha\in(0,\alpha_N)$. 
We proceed as in the proof of Proposition \ref{thm-proof-i-iii}, and 
we find a unique $\beta_\ast(t) > 0$ such that 
$w_t := \beta_\ast (t) v_t$ satisfies \eqref{beta*t-formula}, where $v_t(x) = t^{1/N} v( t^{1/N} x )$. 
Note that $w_1=v$ since $v_1=v$ and $\beta_*(1)=1$. 
Recalling \eqref{38}, we obtain 
\begin{align}\label{beta'1-value}
\beta_*'(1)=-\frac{\frac{a}{N}\|\nabla v\|_N^a}{a\|\nabla v\|_N^a+b\|v\|_N^b}, 
\end{align}
where we used $\beta_*(1)=1$.

	Next, we set
	\[
		g(t) := \int_{\RN} \Phi_N \left( \alpha |w_t|^{N'} \right) d x
	\]
and compute $g'(t)$. Since $\| w_t \|_p^p = \beta_\ast^p(t) t^{-1+p/N} \| v \|_p^p$, 
we observe that
	\[
		g(t)
		= \sum_{j=N-1}^\infty\frac{\alpha^j}{j!}\|w_t\|_{N'j}^{N'j}
		= \sum_{j=N-1}^\infty\frac{\alpha^j}{j!}\beta_*(t)^{N'j}t^{\frac{j}{N-1}-1}\|v\|_{N'j}^{N'j}
	\]
and that 
	\[
		\begin{aligned}
			&\frac{d}{dt} \left(\frac{\alpha^j}{j!} 
			\beta_\ast (t)^{N'j} t^{ \frac{j}{N-1}-1 } \| v \|_{N'j}^{N'j}\right) 
			\\
			=& \left\{ N' j \beta_\ast'(t) \frac{\alpha^j}{j!} \beta_\ast(t)^{N'j-1}  t^{ \frac{j}{N-1}-1 } 
			+ \left( \frac{j}{N-1} -1 \right) \frac{\alpha^j}{j!} \beta_\ast(t)^{N'j} 
			 t^{ \frac{j}{N-1}-2 }  \right\} \| v \|_{N'j}^{N'j}.
		\end{aligned}
	\]
Since $\alpha < \alpha_N$ and $\| \nabla w_t \|_{N} \leq 1$, 
for $\e_0>0$ with $(1+\e_0) \alpha < \alpha_N$, it follows that 
	\[
		\sum_{j=N-1}^\infty \frac{\left\{ \alpha (1+\e_0) \right\}^j}{j!} 
		\beta_\ast(t)^{N'j} t^{ \frac{j}{N-1}-1 } \| v \|_{N'j}^{N'j}
		= \int_{\RN} \Phi_N \left( (1+\e_0) \alpha |w_t|^{N'}  \right) dx < \infty.
	\]
Noting that $\beta_\ast(t) \in C((0,\infty),\R)$ and 
	\[
		\begin{aligned}
			&\left|N' j \beta_\ast'(t) \frac{\alpha^j}{j!} \beta_\ast(t)^{N'j-1}  t^{ \frac{j}{N-1}-1 } 
			+ \left( \frac{j}{N-1} -1 \right) \frac{\alpha^j}{j!} \beta_\ast(t)^{N'j} 
			 t^{ \frac{j}{N-1}-2 } \right|
			 \\
			 =& \left|  N' \frac{j}{(1+\e_0)^j} \frac{\beta_\ast'(t)}{\beta_\ast(t)} 
			 + \left( \frac{j}{N-1} -1 \right) \frac{1}{(1+\e_0)^j t}   \right| 
			 \frac{(1+\e_0)^j \alpha^j  }{j!} \beta_\ast(t)^{N'j} t^{ \frac{j}{N-1}-1 }
			 \\
			 \leq & C_t \frac{(1+\e_0)^j \alpha^j  }{j!} \beta_\ast(t)^{N'j} t^{ \frac{j}{N-1}-1 },
		\end{aligned}
	\]
we have 
	\begin{equation}\label{33}
		\begin{aligned}
			g'(t) &= \sum_{j=N-1}^\infty \frac{d}{dt} \left(\frac{\alpha^j}{j!} 
			\beta_\ast (t)^{N'j} t^{ \frac{j}{N-1}-1 } \| v \|_{N'j}^{N'j}\right) 
			\\
			&=
			\sum_{j=N-1}^\infty 
			\left\{ N' j \beta_\ast'(t) \frac{\alpha^j}{j!} \beta_\ast(t)^{N'j-1}  t^{ \frac{j}{N-1}-1 } 
			+ \left( \frac{j}{N-1} -1 \right) \frac{\alpha^j}{j!} \beta_\ast(t)^{N'j} 
			 t^{ \frac{j}{N-1}-2 }  \right\} \| v \|_{N'j}^{N'j}.
		\end{aligned}
	\end{equation}

	Now recall that $v$ is a maximizer for $D_{N,\alpha}(a,b)$ and $w_1=v$, 
which implies $g(1) \geq g(t)$ and $g'(1) = 0$. Therefore, 
from \eqref{33} and the facts $\beta_*(1)=1$ and $\beta_*'(1)<0$ 
by \eqref{beta'1-value}, it follows that 
\begin{equation}\label{contradic-one}
\begin{aligned}
0&=g'(1)\\
&=\frac{N\alpha^{N-1}}{(N-1)!}\beta_*'(1)\|v\|_N^N
+\sum_{j=N}^\infty\frac{\alpha^j}{j!}\left(
N'j\beta_*'(1)+\frac{j}{N-1}-1
\right)\|v\|_{N'j}^{N'j}\\
&\leq\frac{N\alpha^{N-1}}{(N-1)!}\beta_*'(1)\|v\|_N^N+\frac{1}{N-1}\sum_{j=N}^\infty\frac{\alpha^j}{(j-1)!}\|v\|_{N'j}^{N'j}\\
&=\frac{\alpha^{N-1}}{N-1}\|v\|_N^N\|\nabla v\|_N^a\left(
\frac{N\beta_*'(1)}{(N-2)!\|\nabla v\|_N^a}+\frac{\alpha}{\|v\|_N^N\|\nabla v\|_N^a}\sum_{j=N}^\infty\frac{\alpha^{j-N}}{(j-1)!}\|v\|_{N'j}^{N'j} 
\right). 
\end{aligned}
\end{equation}
Furthermore, applying the Gagliardo-Nirenberg inequality (see \cite{KOS-92,O2})
\begin{align*}
\|v\|_{N'j}^{N'j}\leq C^j j^j \|v\|_N^N\|\nabla v\|_N^{N'j-N}, 
\end{align*}
where $C$ depends only on $N$ , we observe that 
\begin{equation}\label{contradic-two}
\begin{aligned}
\frac{\alpha}{\|v\|_N^N\|\nabla v\|_N^a}\sum_{j=N}^\infty\frac{\alpha^{j-N}}{(j-1)!}\|v\|_{N'j}^{N'j}
&\leq\frac{\alpha}{\|v\|_N^N\|\nabla v\|_N^a}\sum_{j=N}^\infty\frac{\alpha^{j-N}}{(j-1)!}C^j j^j\|v\|_N^N\|\nabla v\|_N^{N'j-N}\\
&=\alpha\sum_{j=N}^\infty\frac{\alpha^{j-N}}{(j-1)!}C^j j^j\|\nabla v\|_N^{N'j-N-a}
\\
&\leq\alpha\sum_{j=N}^\infty\frac{\alpha^{j-N}}{(j-1)!}C^j j^j, 
\end{aligned}
\end{equation}
where we used $N'j-N-a\geq N'N-N-a=N'-a\geq 0$ for $j\geq N$ and $\|\nabla v\|_N\leq 1$. 
Moreover, we have 
\begin{align}
\label{contradic-three}\alpha\sum_{j=N}^\infty\frac{\alpha^{j-N}}{(j-1)!}C^j j^j
=C^N\alpha\sum_{j=0}^\infty\frac{(j+N)^{j+N}}{(j+N-1)!}(\alpha C)^j
\leq C^N\alpha\sum_{j=0}^\infty\frac{(j+N)^{j+N}}{(j+N-1)!}\left(\frac{1}{2e}\right)^j=:\tilde C\alpha
\end{align}
for any $\alpha<\min\{\frac{1}{2eC},\alpha_N\}$ since the radius of the convergence of this power series is $\frac{1}{e}$. 
Hence, combining \eqref{contradic-one}, \eqref{contradic-two} and \eqref{contradic-three} together with \eqref{beta'1-value}, we see for any $\alpha<\min\{\frac{1}{2eC},\alpha_N\}$, 
\begin{align*}
0\leq-\frac{a}{(N-2)!(a\|\nabla v\|_N^a+b\|v\|_N^b)}+\tilde C\alpha, 
\end{align*}
and thus a contradiction occurs provided that 
\begin{align*}
\alpha<\frac{a}{\tilde C(N-2)!(a\|\nabla v\|_N^a+b\|v\|_N^b)}. 
\end{align*}
Here, we observe 
\begin{align*}
\frac{a}{a\|\nabla v\|_N^a+b\|v\|_N^b}=
\begin{cases}
&\frac{a}{b-(b-a)\|\nabla v\|_N^a}\geq\frac{a}{b}\quad\text{if \,}a\leq b,\\
&\frac{a}{a-(a-b)\|v\|_N^b}>1\quad\text{if \,}a>b.
\end{cases}
\end{align*}
Finally, defining $\alpha_0:=\min\left\{
\frac{\min\{\frac{a}{b},1\}}{\tilde C(N-2)!},\frac{1}{2eC},\alpha_N
\right\}$, 
we have a contradiction for all $\alpha<\alpha_0$. 
The proof of Proposition \ref{small-non-exist} is now complete.
\end{proof}

	Next, we give conditions in order for $\alpha_\ast < \alpha_N$ to happen 
when $a \leq N'$. 

	\begin{prop}\label{034}
	Let $\hat{\alpha} \in (0,\alpha_N)$ and $\hat{b} > N^2 / (\hat{\alpha} B_{GN})$. 
	Then there exists an $a_0 = a_0(\hat{\alpha}, \hat{b}, N) \in (0,N')$ such that 
	$D_{N,\alpha} (a,b)$ is attained 
	for all $(\alpha,a,b) \in A_1 \cup A_2$ 
	where $A_1 := \{ \alpha_N \} \times (a_0,N'] \times [\hat{b},N)$ and $A_2 := [\hat{\alpha},\alpha_N) \times (a_0,N'] \times [\hat{b}, \infty)$. 
	In particular, for $(a,b) \in (a_0,N'] \times [\hat{b}, \infty)$, 
	we have $\alpha_\ast \leq \hat{\alpha} < \alpha_N$. 
	\end{prop}


	\begin{proof}
The argument is similar to that of \cite[Proposition 7.1]{DST-16}. 
Let $\hat{\alpha} \in (0,\alpha_N)$ and $\hat{b} > N^2 / (\hat{\alpha} B_{GN})$. 
From Proposition \ref{>-compact} 
it suffices to find $a_0(\hat{\alpha},\hat{b},N) \red{>}  0$ so that 
$D_{N,\alpha}(a,b) > \alpha^{N-1} / (N-1)!$ holds for each 
$(\alpha,a,b) \in A_1 \cup A_2$ where $A_i$ is defined in the above.

	To this end, we set 
	\[
		J_{N,\alpha}(u) := \sum_{j=N-1}^{N} \frac{\alpha^j}{j!} \| u \|_{N'j}^{N'j} 
		= \frac{\alpha^{N-1}}{(N-1)!} \| u \|_{N}^N + \frac{\alpha^N}{N!} 
		\| u \|_{NN'}^{NN'}.
	\]
It is easy to see that 
	\[
		D_{N,\alpha}(a,b) \geq \sup_{ \| \nabla u \|_N^a + \| u \|_N^b \leq 1 } 
		J_{N,\alpha}(u) =: d_{N,\alpha}(a,b).
	\]
Therefore, it suffices to show that 
	\begin{equation}\label{311}
		d_{N,\alpha}(a,b) > \frac{\alpha^{N-1}}{(N-1)!} 
		\quad \text{for any $(\alpha,a,b) \in A_1 \cup A_2$.
		}
	\end{equation}

	To prove \eqref{311}, we use a similar argument to the proof of 
Proposition \ref{thm-proof-i-iii}. Let $V$ be a maximizer for $B_{GN}$ 
with $\| \nabla V \|_{L^N} = 1 = \| V \|_N$ (see \cite[Lemma 8.4.2]{C-03}). 
For $t \in (0,1)$, define $w_t(x):=t^{1/b}V(x)$. 
Then $w_t$ is also a maximizer for $B_{GN}$ and satisfies 
$\| \nabla w_t \|_N^a = t^{a/b}$ and 
$\| w_t \|_N^b = t \in (0,1)$. 
Setting $w_t^\lambda(x) := \lambda w_t(\lambda x)$ for $\lambda > 0$ and 
noting that 
	\[
		\| \nabla w_t^\lambda \|_{N}^N =\lambda^N  \| \nabla w_t \|_N^N, 
		\quad \| w_t^\lambda \|_p^p =  \lambda^{p-N} \| w_t \|_p^p,
	\]
for every $t \in (0,1)$, we have 
	\[
		1 = \| \nabla w_t^{\lambda(t)} \|_N^a + \| w_t^{\lambda(t)} \|_N^b 
		\quad {\rm where} \ 
		\lambda (t) := t^{-1/b} (1-t)^{1/a} .
	\]

	Now put $W_t(x) := w_t^{\lambda(t)}(x)$ and compute $J_{N,\alpha}(W_t)$. 
Recalling that $w_t$ is a maximizer for $B_{GN}$, 
$\|w_t\|_{N}^b = t$ and $ t^{a/b} =  \| \nabla w_t \|_N^a$, 
we have 
	\[
		\begin{aligned}
			&\| W_t \|_{NN'}^{NN'} = \lambda(t)^{N'} \| w_t \|_{NN'}^{NN'} 
			= B_{GN} \lambda(t)^{N'} t^{N'/b} t^{N/b}  
			= B_{GN}t^{N/b} (1-t)^{N'/a}.
		\end{aligned}
	\]
Hence, 
	\[
		\begin{aligned}
			d_{N,\alpha} (a,b) & \geq J_{N,\alpha}(W_t) \\
			&= \frac{\alpha^{N-1}}{(N-1)!} 
			\left( t^{N/b} + \frac{\alpha}{N} \| W_t \|_{NN'}^{NN'} \right)
			= \frac{\alpha^{N-1}}{(N-1)!} 
			\left( t^{N/b} + \frac{\alpha}{N} 
			B_{GN}t^{N/b} (1-t)^{N'/a}
			 \right) \\
			 &=: \frac{\alpha^{N-1}}{(N-1)!} g_{\alpha,a,b}(t).
		\end{aligned}
	\]
Remark that $g_{\alpha,a,b}(1) = 1$ and 
$g_{\alpha,a,b}(t)$ is increasing in $\alpha$, $a$ and $b$ due to $t \in [0,1]$. 
We first show $g_{\alpha,a,b}'(1) < 0$ when 
$(\alpha,a,b) = (\hat{\alpha}, N' , \hat{b})$ 
with $\hat{b} > N^2 / (\hat{\alpha} B_{GN})$.

	When $(\alpha,a,b) = (\hat{\alpha}, N' , \hat{b})$, a simple computation gives 
	\[
		g_{\hat{\alpha}, N',\hat{b}}'(1) = \frac{N}{\hat{b}} - \frac{\hat{\alpha}}{N} B_{GN},
	\]
and then we have $g_{\hat{\alpha}, N',\hat{b}}' (1) < 0$ 
since $\hat{b} > N^2/(\hat{\alpha} B_{GN})$. 
Note that $g_{\hat{\alpha}, N',\hat{b}}'(1) < 0$ implies 
$1 < \max_{0 \leq t \leq 1} g_{\hat{\alpha}, N',\hat{b}} (t)$. 
Since the map $ a \mapsto g_{\hat{\alpha}, a,\hat{b}}(t) : (0,N'] \to C([0,1])$ is continuous, 
we can find an $a_0 >0$, depending only on $\hat{\alpha}, \hat{b} , N$, such that 
	\[
		1 < \max_{0 \leq t \leq 1} g_{\hat{\alpha}, a,\hat{b}} (t) \quad 
		{\rm for\ all} \ a \in (a_0,N'] .
	\]
Recalling that $g_{\alpha,a,b }(t)$ is increasing in $\alpha$ and $b$, 
we observe that 
	\[
		1 < \max_{0 \leq t \leq 1} g_{\alpha,a,b}(t) \quad 
		{\rm for\ every} \ (\alpha,a,b) \in A_1 \cup A_2.
	\]
Therefore, when $(\alpha,a,b) \in A_1 \cup A_2$, we obtain 
	\[
		\frac{\alpha^{N-1}}{(N-1)!} < \frac{\alpha^{N-1}}{(N-1)!} 
		\max_{0 \leq t \leq 1} g_{\alpha, a,b} (t) 
		\leq d_{N,\alpha}(a,b) \leq D_{N,\alpha}(a,b).
	\]
Noting that Proposition \ref{>-compact} is applicable for 
each $(\alpha,a,b) \in A_1 \cup A_2$, we see that 
that $D_{N,\alpha}(a,b)$ is attained and $\alpha_\ast \leq \hat{\alpha} < \alpha_N$. 
	\end{proof}

We are now in a position to prove Theorem \ref{alpha*estimates}. 

\begin{proof}[Proof of Theorem \ref{alpha*estimates}]
The statements (i), (ii) and (iii) follow from Propositions \ref{thm-proof-i-iii}, 
\ref{small-non-exist} and \ref{034}, respectively. 
	\end{proof}

\appendix

\section{Proof of $N^2/(\alpha_N B_{GN}) < N$}
\label{appendix}
\noindent

	The purpose of this appendix is to prove 
	\begin{equation}\label{a1}
		\frac{N^2}{\alpha_N B_{GN}} < N
	\end{equation}
for all $N \geq 2$. 
Notice that when $N =2$, 
\eqref{a1} is equivalent to $2/ B_{GN} < \alpha_2 = 4 \pi$ and 
this is already known (see \cite{B, I, W}). 
However, we also provide a simple proof of this fact below.

	We first rewrite \eqref{a1}. 
Using the Schwarz rearrangement, it follows that 
	\[
		B_{GN} = \sup_{u \in W^{1,N}_{\rm r} (\RN) \setminus \{0\} } 
		\frac{ \| u \|_{NN'}^{NN'} }{ \| u \|_{N}^N \| \nabla u \|_{N}^{NN'-N} }, 
	\]
where $W^{1,N}_{\rm r}(\RN)$ is a set consisting of radial functions in $W^{1,N}(\RN)$. 
For $ u \in W^{1,N}_{\rm r}(\RN)$, one has
	\[
		\| \nabla u \|_N^N = \omega_{N-1} \int_0^\infty r^{N-1} |u'(r) |^{N} d r, \quad 
		\| u \|_p^p = \omega_{N-1} \int_0^\infty r^{N-1} |u(r)|^p dr,
	\]
which implies 
	\[
		B_{GN} = \omega_{N-1}^{-1/(N-1)} 
		 \sup_{u \in W^{1,N}_{\rm r} (\RN) \setminus \{0\} } 
		\frac{\int_0^\infty r^{N-1}|u(r)|^{NN'} dr }
		{\int_0^\infty r^{N-1} |u(r)|^N dr 
		\left( \int_0^\infty r^{N-1} |u'(r)|^N dr \right)^{1/(N-1)}
		 }.
	\]
Recalling $\alpha_N = N \omega_{N-1}^{1/(N-1)}$, we see that \eqref{a1} is equivalent to 
	\begin{align}
			1 - \inf_{ u \in W^{1,N}_{\rm r} (\RN) \setminus \{0\}  }
			\frac{
			\int_0^\infty r^{N-1} |u(r)|^N dr 
				\left( \int_0^\infty r^{N-1} |u'(r)|^N dr \right)^{1/(N-1)}
			}
			{
			\int_0^\infty r^{N-1}|u(r)|^{NN'} dr
			}
			>0. 
			\label{a2}
	\end{align}
Therefore, instead of \eqref{a1}, we shall prove \eqref{a2} for every $N \geq 2$.

	When $N=2$, we can check \eqref{a2} by setting 
$u(r) := \max\{ 0 , (1-r)^3 \}$. In fact, since
	\[
		\begin{aligned}
			&\int_0^1 r (1-r)^6 dr = B(2,7) = \frac{\Gamma(2) \Gamma(7)}{\Gamma(9)}, \quad 
			9 \int_0^1 r (1-r)^4 dr = 9 B(2,5) = 9 \frac{\Gamma(2) \Gamma(5)}{\Gamma(7)},
			\\
			& \int_0^1 r (1-r)^{12} dr = B(2,13) 
			= \frac{\Gamma(2)\Gamma (13)}{\Gamma(15)}
		\end{aligned}
	\]
where $B(x,y)$ and $\Gamma (z)$ are the beta function and 
the gamma function, we see 
	\[
		\begin{aligned}
			& \frac{
			\int_0^\infty r^{N-1} |u(r)|^N dr 
				\left( \int_0^\infty r^{N-1} |u'(r)|^N dr \right)^{1/(N-1)}
			}
			{
			\int_0^\infty r^{N-1}|u(r)|^{NN'} dr
			}
			\\
			=& \frac{\Gamma(7)}{\Gamma(9)} 
			9 \frac{\Gamma(5)}{\Gamma(7)} \frac{\Gamma(15)}{\Gamma(13)} 
			= 9 \frac{1}{8\cdot 7 \cdot 6 \cdot 5}  14 \cdot 13 
			= \frac{39}{40} < 1.
		\end{aligned}
	\]

	For the case $N \geq 3$, we choose $u(r) = e^{-r}$ as a test function. 
Observe that 
	\[
\begin{aligned}
		&\int_0^\infty r^{N-1} e^{-Nr} dr 
		= \frac{\Gamma(N)}{N^N} 
		= \int_0^\infty r^{N-1} | (e^{-r})' |^N d r, \\
		&\int_0^\infty r^{N-1} 
		e^{-N^2 r/ (N-1)} dr 
		= \left( \frac{N-1}{N^2} \right)^N \Gamma(N).
\end{aligned}
	\]
Therefore, we obtain 
	\begin{equation}\label{a02}
		\begin{aligned}
			&\frac{
			\int_0^\infty r^{N-1} |u(r)|^N dr 
			\left( \int_0^\infty r^{N-1} |u'(r)|^N dr \right)^{1/(N-1)}
			}
			{
			\int_0^\infty r^{N-1}|u(r)|^{NN'} dr
			}\\
			= & \,
			\Gamma(N)^{1/(N-1)} (N-1)^{-N} 
			N^{(N^2-2N)/(N-1)}
			=: C_N.
		\end{aligned}
	\end{equation}
Thus to prove \eqref{a2}, it suffices to show $C_N^{N-1} < 1$ for $N\geq 3$, 
which is equivalent to $(N-1) \log(C_N) < 0$.

From 
	\[
		C_{N}^{N-1} 
		= (N-1) ! N^{N^2 - 2 N} (N-1)^{-N^2+N},  
	\]
it follows that 
	\[
		\begin{aligned}
			\log (C_N^{N-1}) 
			&= \sum_{k=1}^{N-1} \log k 
			+ (N^2 - 2N) \log N - (N^2-N) \log(N-1)
			\\
			&=\sum_{k=1}^{N-1} \log k 
			+ (N^2  - N) \left( \log N - \log(N-1) \right) - N \log N
			\\
			&=\left\{ \sum_{k=1}^{N-1} \log k - N 
			\left( \log N - 1\right)   \right\} 
			+ \left\{ - N 
			+ N(N-1) \log \left( 1 + \frac{1}{N-1}  \right) \right\}
			\\
			&=: d_N+e_N
		\end{aligned}
	\]

\noindent
{\bf Claim 1:} {\sl For each $N \geq 3$, there holds
	\[
		e_N < - \frac{1}{2}.
	\]
}

\smallskip

	In fact, by $\log(1+x) = \sum_{k=1}^\infty (-1)^{k-1} x^k / k$ for $ 0 \leq x <1$ and 
$N \geq 3$, we have 
	\[
		\begin{aligned}
			&\log \left( 1 + \frac{1}{N-1} \right) \\
			&= \frac{1}{N-1} - \frac{1}{2(N-1)^2} + \frac{1}{3(N-1)^3} 
			+ \sum_{k=2}^\infty \left( -\frac{1}{2k} \frac{1}{(N-1)^{2k}} 
			+ \frac{1}{2k+1} \frac{1}{(N-1)^{2k+1}} \right)
			\\
			&< \frac{1}{N-1} - \frac{1}{2(N-1)^2} + \frac{1}{3(N-1)^3}.
		\end{aligned}
	\]
Hence, one sees that  
	\[
		\begin{aligned}
			e_N 
			&< - N 
			+ N (N-1) \left\{ \frac{1}{N-1} - \frac{1}{2(N-1)^2} + \frac{1}{3(N-1)^3} \right\}
			\\
			&= - \frac{1}{2} \frac{N}{N-1} + \frac{N}{3(N-1)^2} 
			= -\frac{1}{2} + \frac{3-N}{6(N-1)^2} \leq -\frac{1}{2}.
		\end{aligned}
	\]

\noindent
{\bf Claim 2:} There holds {\sl $d_{N+1}< d_{N}$ for every $N \geq 3$}. 

\smallskip

Since 
	\[
		\begin{aligned}
			d_{N} - d_{N+1} 
			&= \sum_{k=1}^{N-1} \log k - N \left( \log N - 1 \right) 
			- \sum_{k=1}^{N} \log k 
			+ (N+1) 
			\left\{ \log(N+1) - 1  \right\}
			\\
			&= (N+1) \left\{ \log(N+1) - \log N  \right\} - 1,
		\end{aligned}
	\]
set $f(x) := (x+1) \{ \log(x+1) - \log(x) \} - 1$. Noting that 
	\[
		\begin{aligned}
			f'(x) &= \log(x+1) - \log(x)  + 1 - \frac{x+1}{x} 
			= \log(x+1) - \log(x) - \frac{1}{x} = \log(1+x^{-1}) - \frac{1}{x},\\
			f''(x) &= \frac{1}{x+1} - \frac{1}{x} + \frac{1}{x^2} 
			= \frac{1}{x^2(x+1)} > 0,
		\end{aligned}
	\]
and that $f'(x) \to 0$ as $x \to \infty$, we infer that 
$f'(x) < 0$ for all $x \geq 1$. Since 
	\[
		f(x) = (x+1) \log( 1 + x^{-1} ) - 1 
		= (x+1) \left(x^{-1} - \frac{x^{-2}}{2} + O(x^{-3}) \right)
		-1 = \frac{x^{-1}}{2} + O(x^{-2}),
	\]
we have $f(x) \to 0$ as $x \to \infty$, and 
$f(x) > 0$ for all $x \geq 1$. Thus $f(N) > 0$ and Claim 2 holds. 

\medskip

\noindent
{\bf Claim 3:} There holds {\sl $\log(C_N^{N-1}) < 0$ for every $N \geq 3$. }

\smallskip

By Claims 1-2, it is enough to prove $d_3 < 1/2$. 
Since $d_3 = \log 2 -3 ( \log 3 - 1 )$, we see that $d_3<\frac{1}{2}$ is equivalent to $e^5<\frac{729}{4}$. 
Moreover, we observe that 
\[
		e^5 < (2.8)^5 = 172.10368 < 182.25 = \frac{729}{4}. 
	\]
\noindent
Hence, Claim 3 is proved. 

\medskip

From \eqref{a2}, \eqref{a02} and Claim 3, we get the desired inequality 
$N^2 / (\alpha_N B_{GN}) < N$.

\subsection*{Acknowledgement}
This work was supported by JSPS KAKENHI Grant Number JP16K17623. 
The authors would like to express their hearty thanks to the referees for their
valuable comments.


\begin{thebibliography}{99}
\bibitem{AT}S. Adachi and K. Tanaka, 
\emph{A scale-invariant form of Trudinger-Moser inequality and its best exponent}, 
Proc. Amer. Math. Soc. {\bf 1102} (1999), 148--153.


\bibitem{B}W. Beckner, 
\emph{Estimates on Moser embedding},
Potential Anal. {\bf 20} (2004), 345--359.

\bibitem{BL-83-1}H. Berestycki and P.-L. Lions, 
\emph{Nonlinear scalar field equations. I. Existence of a ground state},  
Arch. Rational Mech. Anal. {\bf 82} (1983), 313--345.

\bibitem{BJT-08}J. Byeon, L. Jeanjean and K. Tanaka, 
\emph{Standing waves for nonlinear Schr\"odinger equations with a general nonlinearity : one and two dimensional cases}, 
Comm. Partial Differential Equations {\bf 33} (2008), 1113--1136. 

\bibitem{C}D. M. Cao, 
\emph{Nontrivial solution of semilinear elliptic equation with critical exponent in $\Bbb R^2$}, 
Comm. Partial Differential Equations {\bf 17} (1992), 407--435. 

\bibitem{CC}L. Carleson and S.-Y. A. Chang,
\emph{On the existence of an extremal function for an inequality of J. Moser}, 
Bull. Sci. Math. (2) {\bf 110} (1986), 113--127. 

\bibitem{CST-14}D. Cassani, F. Sani and C. Tarsi, 
\emph{Equivalent Moser type inequalities in $\mathbb{R}^2$ and the zero mass case}, 
J. Funct. Anal. {\bf 267} (2014), 4236--4263.

\bibitem{C-03}T. Cazenave, 
\emph{Semilinear Schr\"odinger equations. Courant Lecture Notes in Mathematics {\bf 10}}, 
New York University, Courant Institute of Mathematical Sciences, New York ; American Mathematical Society, Providence, RI, 2003.


\bibitem{D}J.M. Do \'O, 
\emph{$N$-Laplacian equations in $\Bbb R^N$ with critical growth}, 
Abstr. Appl. Anal. {\bf 2} (1997), 301--315.

\bibitem{DST-16}J.M. Do \'O, F. Sani and C. Tarsi, 
\emph{Vanishing-concentration-compactness alternative for the Trudinger-Moser inequality in $\mathbb{R}^N$}, 
Commun. Contemp. Math. (2016), 1650036 (27 pages). 

\bibitem{F}M. Flucher, 
\emph{Extremal functions for the Trudinger-Moser inequality in $2$ dimensions}, 
Comm. Math. Helv. {\bf 67} (1992), 471--479.

\bibitem{I}M. Ishiwata, 
\emph{Existence and nonexistence of maximizers for variational problems associated with Trudinger-Moser type inequalities in $\Bbb R^N$}, 
Math. Ann. {\bf 351} (2011), 781--804.

\bibitem{INW}M. Ishiwata, M. Nakamura and H. Wadade, 
\emph{On the sharp constant for the weighted Trudinger-Moser type inequality of the scaling invariant form}, 
Ann. Inst. H. Poincar\'e Anal. Non Lin\'eaire {\bf 31} (2014), 297--314.

\bibitem{IW-16}
M. Ishiwata and H. Wadade, 
\emph{On the effect of equivalent constraints on a maximizing problem associated 
with the Sobolev type embeddings in $\mathbb{R}^N$}.  
Math. Ann. {\bf 364} (2016), no. 3-4, 1043--1068.

\bibitem{KOS-92}H. Kozono, T. Ogawa and H. Sohr, 
\emph{Asymptotic behaviour in Lr for weak solutions of the Navier-Stokes equations in exterior domains}, 
Manuscripta Math. {\bf 74} (1992), 253--275. 

\bibitem{KSW}H. Kozono, T. Sato and H. Wadade, 
\emph{Upper bound of the best constant of a Trudinger-Moser inequality and its application to a Gagliardo-Nirenberg inequality}, 
Indiana Univ. Math. J. {\bf 55} (2006), 1951--1974. 

\bibitem{Lam}N. Lam, 
\emph{Maximizers for the singular Trudinger-Moser inequalities in the subcritical cases}, 
Proc. Amer. Math. Soc. {\bf 145} (2017), 4885--4892.

\bibitem{LL}N. Lam and G. Lu, 
\emph{Sharp singular Adams inequalities in high order Sobolev spaces}, 
Methods Appl. Anal. {\bf 19} (2012), 243--266. 

\bibitem{LLZ}N. Lam, G. Lu and L. Zhang, 
\emph{Equivalence of critical and subcritical sharp Trudinger-Moser-Adams inequalities}, 
to appear in Rev. Mat. Iberoam. (2017).  

\bibitem{LR}Y. Li and B. Ruf, 
\emph{A sharp Trudinger-Moser type inequality for unbounded domains in $\Bbb R^n$}, 
Indiana Univ. Math. J. {\bf 57} (2008), 451--480.

\bibitem{LY-18}
X.  Li and Y. Yang, 
\emph{Extremal functions for singular Trudinger-Moser inequalities in the entire Euclidean space}, 
J. Differential Equations, in press (arXiv:1612.08247). 

\bibitem{L}K. C. Lin, 
\emph{Extremal functions for Moser's inequality}, 
Trans. Amer. Math. Soc. {\bf 348} (1996), 2663--2671.

\bibitem{M}J. Moser, 
\emph{A sharp form of an inequality by N. Trudinger}, 
Indiana Univ. Math. J. {\bf 20} (1970/1971), 1077--1092.

\bibitem{NW}S. Nagayasu and H. Wadade, 
\emph{Characterization of the critical Sobolev space on the optimal singularity at the origin},
J. Funct. Anal. {\bf 258} (2010), 3725--3757.

\bibitem{Og}T. Ogawa,
\emph{A proof of Trudinger's inequality and its application to nonlinear Schrodinger equation}, 
Nonlinear Anal. {\bf 14}, (1990) 765--769.

\bibitem{OgOz}T. Ogawa and T. Ozawa,
\emph{Trudinger type inequalities and uniqueness of weak solutions for the nonlinear Schrodinger mixed problem}, 
J. Math. Anal. Appl. {\bf 155}, (1991) 531--540.

\bibitem{O}T. Ozawa, 
\emph{Characterization of Trudinger's inequality}, 
J. Inequal. Appl. {\bf 1} (1997), 369--374.

\bibitem{O2}T. Ozawa,
\emph{On critical cases of Sobolev's inequalities}, 
J. Funct. Anal. {\bf 127} (1995), 259--269.

\bibitem{Po}S. I. Pohozaev,
\emph{The Sobolev embedding in the case $pl=n$}, 
Proceedings of the Technical Scientific Conference on Advances of Scientific Research 1964--1965. 
Mathematics Section, pp.158--170, Moskov. Energetics Inst., Moscow, (1965).

\bibitem{R}B. Ruf,
\emph{A sharp Trudinger-Moser type inequality for unbounded domains in $\Bbb R^2$}, 
J. Funct. Anal. {\bf 219} (2005), 340--367.

\bibitem{S-77}
W. A. Strauss, 
\emph{Existence of solitary waves in higher dimensions},  
Comm. Math. Phys. {\bf 55} (1977), 149--162.

\bibitem{S}M. Struwe, 
\emph{Critical points of embeddings of $W^{1,N}_0$ into Orlicz spaces}, 
Ann. Inst. H. Poincar\'e Anal. Non Lin\'eaire {\bf 5} (1988), 425--464.


\bibitem{T}N. S. Trudinger, 
\emph{On imbeddings into Orlicz spaces and some applications}, 
J. Math. Mech. {\bf 17} (1967), 473--483.

\bibitem{W}M. I. Weinstein, 
\emph{Nonlinear Schrodinger equations and sharp interpolation estimates}, 
Commun. Math. Phys. {\bf 87} (1982/1983), 567--576.

\bibitem{Y}V. I. Yudovich, 
\emph{Some estimates connected with integral operators and with solutions of elliptic equations}, 
Dok. Akad. Nauk SSSR {\bf 138} (1961), 805--808. 

\end{thebibliography}
\end{document}